\documentclass[11pt,twoside,a4paper]{article}
\usepackage[a4paper,centering,bindingoffset=0cm,marginpar=2cm,margin=2.5cm]{geometry}

\usepackage[utf8]{inputenc}
\usepackage[T1]{fontenc}
\usepackage{microtype}
\usepackage{amsmath,amssymb,amsthm}
\numberwithin{equation}{section}
\usepackage{lmodern}
\usepackage[dvipsnames]{xcolor}
\usepackage{siunitx}
\usepackage{aliascnt}

\usepackage{mathtools}

\usepackage{graphicx,tikz,pgfplots}
\pgfplotsset{compat=newest}
\graphicspath{{images/}}

\usepackage{subcaption}
\usepackage{sidecap}
\mathtoolsset{showonlyrefs}

\usepackage[pdftex,colorlinks=true,linkcolor=blue,citecolor=green,urlcolor=blue,bookmarks=true,bookmarksnumbered=true,pdfusetitle]{hyperref}

\usepackage[skip=0.5\baselineskip plus 2pt minus 1pt,parfill=10pt]{parskip}

\makeatletter
\def\thm@space@setup{%
  \thm@preskip=\parskip \thm@postskip=0pt
}

\newtheorem{lemma}{Lemma}[section]

\newaliascnt{proposition}{lemma}

\aliascntresetthe{proposition}

\newaliascnt{corollary}{lemma}

\aliascntresetthe{corollary}

\newaliascnt{theorem}{lemma}
\newtheorem{theorem}[theorem]{Theorem}
\aliascntresetthe{theorem}

\newaliascnt{definition}{lemma}

\aliascntresetthe{definition}

\newaliascnt{assumption}{lemma}

\aliascntresetthe{assumption}

\newaliascnt{notation}{lemma}

\aliascntresetthe{notation}

\newaliascnt{example}{lemma}

\aliascntresetthe{example}

\newaliascnt{experiment}{lemma}

\aliascntresetthe{experiment}

\newaliascnt{remark}{lemma}

\aliascntresetthe{remark}

\usepackage[ruled]{algorithm2e}

\newcommand{\N}{\mathbb{N}}
\newcommand{\R}{\mathbb{R}}

\newcommand{\C}{\mathbb{C}}

\newcommand{\SO}{\mathrm{SO}(3)}

\newcommand{\abs}[1]{\left|#1\right|}
\newcommand{\norm}[1]{\left\|#1\right\|}

\newcommand{\inner}[2]{\left<#1,#2\right>}
\newcommand{\inn}[1]{\left<#1\right>}
\newcommand{\ktran}{\mathcal{F}}

\newcommand{\e}{\mathrm e}

\renewcommand{\i}{\mathrm i}
\newcommand{\dd}{\, \mathrm{d}}
\renewcommand{\Re}{\operatorname{Re}}

\newcommand{\zb}[1]{\boldsymbol{#1}}

\newcommand{\bc}{{\boldsymbol c}}
\newcommand{\bd}{{\boldsymbol d}}
\newcommand{\be}{{\boldsymbol e}}

\newcommand{\bh}{{\boldsymbol h}}

\newcommand{\bk}{{\boldsymbol k}}

\newcommand{\bx}{{\boldsymbol x}}
\newcommand{\bxh}{{\boldsymbol{\hat x}}}

\newcommand{\bgam}{{\boldsymbol \gamma}}

\newcommand{\cI}{\mathcal{I}}

\newcommand{\ui}{u^{\mathrm{inc}}}
\newcommand{\us}{u^{\mathrm{sca}}}
\newcommand{\utot}{u^{\mathrm{tot}}}

\newcommand{\rM}{r_{\mathrm{M}}}

\title{Application and Evaluation of the Common Circles Method}
\author{
Michael Quellmalz\footnotemark[1]
\and 
Mia Kvåle Løvmo\footnotemark[2]
\and 
Simon Moser\footnotemark[2]
\and 
Franziska Strasser\footnotemark[2]
\and
Monika Ritsch-Marte\footnotemark[2]
}

\date{}

\begin{document}

\maketitle

\footnotetext[1]{
TU Berlin,
\url{quellmalz@math.tu-berlin.de}
} 
\footnotetext[2]{
Medical University of Innsbruck,
\url{Loevmo.Mia@i-med.ac.at}, 
\url{Simon.Moser@i-med.ac.at},
\url{Monika.Ritsch-Marte@i-med.ac.at}
}

\begin{abstract}

We investigate the application of the common circle method for estimating sample motion in optical diffraction tomography (ODT) of sub-millimeter sized biological tissue.  
When samples are confined via contact-free acoustical force fields, their motion must be estimated from the captured images. 
The common circle method identifies intersections of Ewald spheres in Fourier space to determine rotational motion.
This paper presents a practical implementation, incorporating temporal consistency constraints to achieve stable reconstructions.  
Our results on both simulated and real-world data demonstrate that the common circle method provides a computationally efficient alternative to full optimization methods for motion detection.


\medskip
\noindent
\textit{Math Subject Classifications.}
92C55, 
78A46, 
94A08, 
42B05.  
\end{abstract}

\section{Introduction}
\noindent

We consider the tomographic imaging of sub-millimeter sized biological samples by means of \emph{optical diffraction tomography (ODT)} \cite{Dev82,KakSla01,sung2009optical}.
In ODT, the object is illuminated from different angles to reconstruct its 3D refractive index.
While traditional sample immobilization, e.g. 
in gel, may restrict 
biological processes,
contact-free methods utilizing optical \cite{JonMarVol15} or acoustical tweezers \cite{dholakia2020comparing,ThaSteMeiHilBerRit11} allow imaging of cells in their natural environment.
However, the motion parameters are not known exactly and must be estimated from the captured images.

In standard computed tomography, 
optical diffraction is negligible because the wavelength of x-rays is much smaller than object features,
allowing for a straight-line light propagation model.
In contrast, our ODT setup involves \si{\micro\meter}-sized biological cells imaged with visible light,
where diffraction effects require more sophisticated wave propagation models.

Therefore, 
the \emph{common line method} \cite{ElbRitSchSchm20,PenZhuFra96,SinCoiSigCheShk10,Hee87,WanSinWen13} for orientation detection in x-ray tomography, which uses that each x-ray image becomes a plane in Fourier space,
is not applicable.
Under the Born approximation,
each ODT measurement corresponds to an Ewald sphere in Fourier space, cf.\ \cite{MueSchuGuc15_report}.
The \emph{common circle method} \cite{ElbQueSchSte23}, see also \cite{BorTeg11}, 
reconstructs the motion of the imaged object by identifying the intersections (common circles) of these sphere, see \autoref{fig:intersection_spheres} left.

While the theoretical foundations of the common circle method were established in~\cite{ElbQueSchSte23},
this paper focuses on the application to real-world data.
We obtain stable reconstructions
by adding additional regularization and ensuring temporal consistency.
Finally, we provide a quantitative comparison 
with the full optimization approach \cite{Mos25} for ODT motion detection, which simultaneously reconstructs object and motion based on the beam propagation method, see also \cite{Kvale2024ultima}.
In comparison, the common circle method is less accurate, 
but it is considerably faster and does not require an initialization of the rotation parameters,
therefore providing a good initial guess for more computationally intensive techniques.

This paper is structured as follows.
\autoref{sec:ODT} outlines the theoretical basis of ODT and the common circle method.
\autoref{sec:recon} describes the reconstruction approach.
\autoref{sec:numerics} presents our numerical data processing and reconstruction.
Conclusions are drawn in \autoref{sec:conclusions}.

\section{Diffraction Tomography}\label{sec:ODT}

In this section, we describe the setup of diffraction tomography with a moving object and the common circle method for reconstructing the rotations,
whose theoretical foundations were derived in \cite{ElbQueSchSte23}.

\subsection{Fourier Diffraction Theorem}

We start with the model of \emph{optical diffraction tomography} for a fixed object, 
see more detail in \cite{KirQueRitSchSet21,KirQueSet25,Wol69}.
The unknown object is illuminated by a \emph{plane wave}
\begin{equation} \label{eq:plane_wave}
  \ui(\bx) \coloneqq \e^{\i k_0 x_3}, \quad \bx\in\R^3,
\end{equation}
which propagates in direction $\be^3 = (0,0,1)^\top
$ with wave number $k_0 = 2\pi n_0 / \lambda_0$,
where $n_0$ is the constant refractive index of the surrounding medium and $\lambda_0$ is the wavelength of the incident field. 
Let $n(\bx)$ denote the \emph{refractive index} at position $\bx \in \R^3$. 
We have $n(\bx)=n_0$ outside the object.
The \emph{scattering potential} or object function
\begin{equation}\label{eq:f}
  f(\bx)\coloneqq 
  k_0^2 \left( \frac{n(\bx)^2}{n_0^2}-1\right) 
\end{equation}
vanishes outside the object.
We assume $f$ is piecewise continuous and compactly supported.
The incident wave $\ui$ induces a \emph{scattered wave} $\us$ 
that solves the partial differential equation
\begin{equation} \label{eq:totII}
-(\Delta + k_0^2)\us(\bx) = f(\bx) \left(\us(\bx)+\ui(\bx)\right)
,\qquad \bx\in\R^3.
\end{equation}
More specifically, $\us$ is the outgoing solution, which fulfills the Sommerfeld radiation condition
\begin{equation}\label{eq:Sommerfeld}
\lim_{r\to\infty} \max_{\norm\bx=r} \norm{\bx} \abs{\langle\nabla\us(\bx),\tfrac{\bx}{\norm\bx}\rangle-\i k_0\us(\bx)}=0.
\end{equation}
If $ f$ is sufficiently small, we neglect $f\us$ on the right-hand side of \eqref{eq:totII} to obtain the \emph{Born approximation} $u$ of the scattered field $\us$, determined by the Helmholtz equation
\begin{equation}
-(\Delta + k_0^2) u(\bx) = f(\bx) \ui(\bx).\label{eq:Born} 
\end{equation}
In the following, we assume the Born approximation to be valid,
which holds for small objects which mildly scatter,
in particular that total phase shift through the object being much less than $2\pi$,
cf.\ \cite{FauKirQueSchSet23,KakSla01,YanYanLinWanQiaHir23}.

We denote the measurements of the scattered wave at the plane $\{\bx\in\R^3:x_3=\rM\}$ at $\rM $ outside the object
by
\begin{equation}
  m(x_1,x_2)\coloneqq u(x_1,x_2,\rM)
  ,\qquad (x_1,x_2)\in \R^d.
\end{equation}
We define the $d$-dimensional \emph{Fourier transform} 
of an integrable function $g\colon\R^d\to\C$ 
by
\begin{equation}\label{eq:FourierDef}
\mathcal F_d[g](\bxh)
  \coloneqq (2\pi)^{-d/2} \int_{\R^d} g(\bx)\,\e^{-\i \langle\bx,\bxh\rangle } \dd \bx
  ,\quad\bxh\in\R^d,
\end{equation}
and the ball of radius $r>0$ by
$$
\mathcal B^d_{r} \coloneqq \{ \bx\in\R^d : \norm{\bx}< r \},\qquad d\in\N,\;r>0.
$$
The \emph{Fourier diffraction theorem} relates the 2D Fourier transform of the measurements~$m$ to the 3D Fourier transform of the scattering potential~$f$, see \cite{KakSla01,KirQueRitSchSet21,NatWue01,Wol69}. 
We have 
\begin{equation} \label{eq:recon}
\ktran_2[m](\bk)=\sqrt{\frac{\pi}{2}}\,\frac{\i \e^{\i\kappa(\bk) \rM}}{\kappa(\bk)}\ktran_3 [f]\left(\bh(\bk) \right)\qquad\forall\,\bk = (k_1,k_2)\in\mathcal B_{k_0}^2,
\end{equation}
where $\bh\colon\mathcal B^2_{k_0} \to\R^3$ is defined by 
\begin{equation}\label{eq:h}
\bh(\bk)
\coloneqq
\begin{pmatrix}\bk\\\kappa(\bk)-k_0\end{pmatrix}, \quad \text{and}\quad \kappa(\bk) \coloneqq \sqrt{k_0^2-\norm{\bk}^2}.
\end{equation}
The left-hand side of \eqref{eq:recon} is the Fourier transform of the measured 2D image $m$,
and the right-hand side is the 3D Fourier transform of $f$ 
evaluated on a hemisphere whose north pole is the origin $\zb0$, see \autoref{fig:intersection_spheres}.

\subsection{Motion of the object}
The object is exposed to a rigid motion
depending on time $t\in[0,T]$, such that the scattering potential of the moving object is
\begin{equation} \label{eq:f_t}
f_t(\bx)
=
f(R_t(\bx-\bd_t))
,\qquad \bx\in\R^3,
\end{equation}
with a rotation matrix 
$$
R_t \in\SO\coloneqq\{Q\in\R^{3\times3}: Q^\top Q = I, \, \mathrm{det} \, Q = 1\},
$$
and a translation vector $\bd_t\in\R^3$. 
The incident wave \eqref{eq:plane_wave} and the measurement plane $\{\bx\in\R^3:x_3=\rM\}$ stay the same as above.
The scattered wave $u_t$ is the solution of \eqref{eq:Born} with $f$ replaced by~$f_t$.
Denoting the measurements by 
\begin{equation}\label{eq:meas}
  m\colon[0,T]\times\R^2 \to\R
  ,\qquad
 m_t(x_1,x_2)\coloneqq u_t(x_1,x_2,\rM),
\end{equation}
the Fourier diffraction theorem~\eqref{eq:recon} becomes \cite{ElbQueSchSte23}
\begin{equation}\label{eq:recon_n}
    \ktran_2[m_t](\bk)
    = \sqrt{\frac{\pi}{2}}\frac{\i \e^{\i\kappa(\bk)\rM}}{\kappa(\bk)}\ktran_3 [f]\left(R_t\bh(\bk) \right)\e^{-\i\inner{\bd_t}{\bh(\bk)}}
    ,\qquad \bk\in \mathcal B^2_{k_0}.
\end{equation}

%
\subsection{Common Circle Method} \label{sec:comm_circ}
%
\begin{figure}[t]
\begin{center}
  \centering\includegraphics[height=11em]{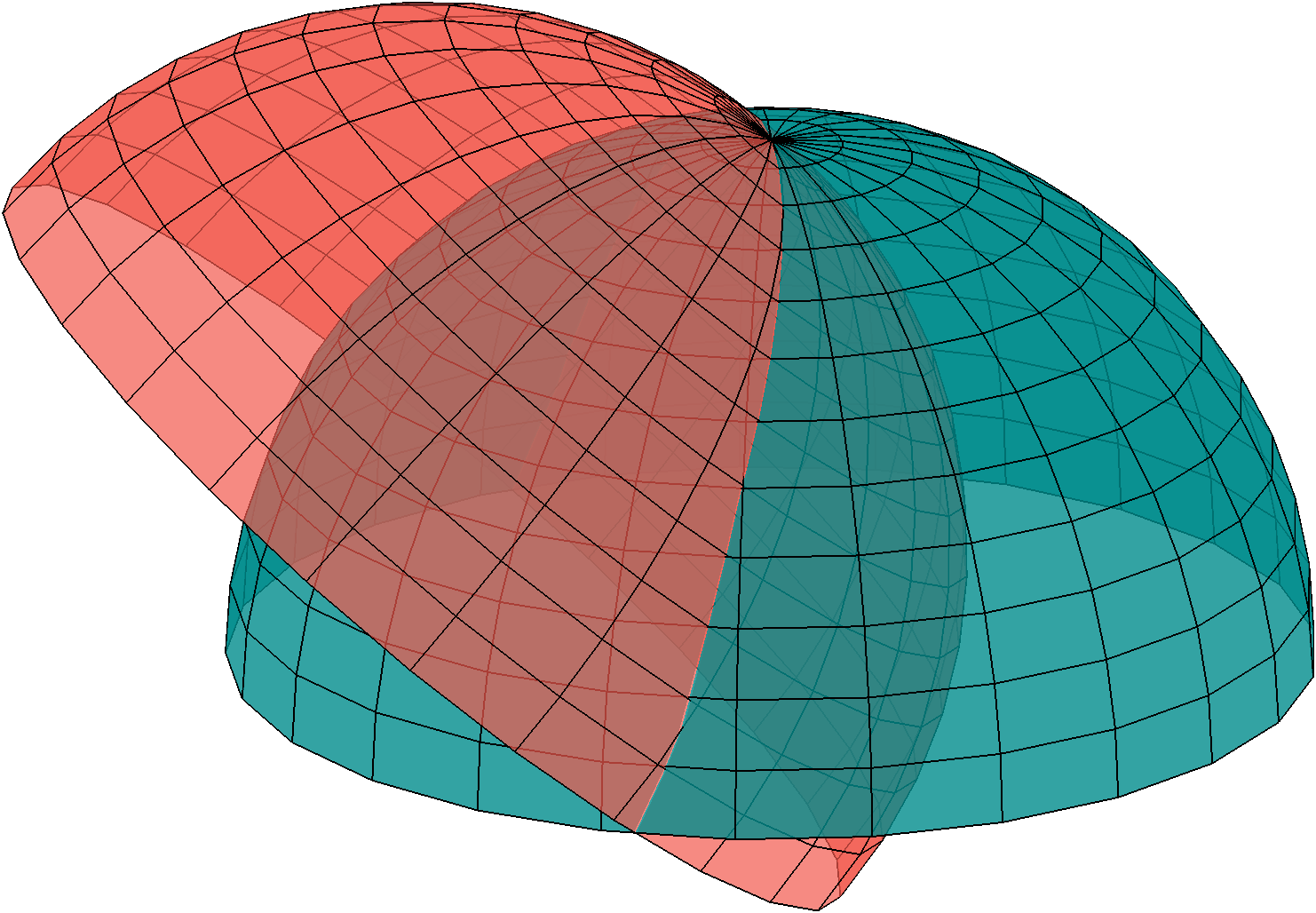}
  \hspace{2cm}
  \centering\includegraphics[height=12em]{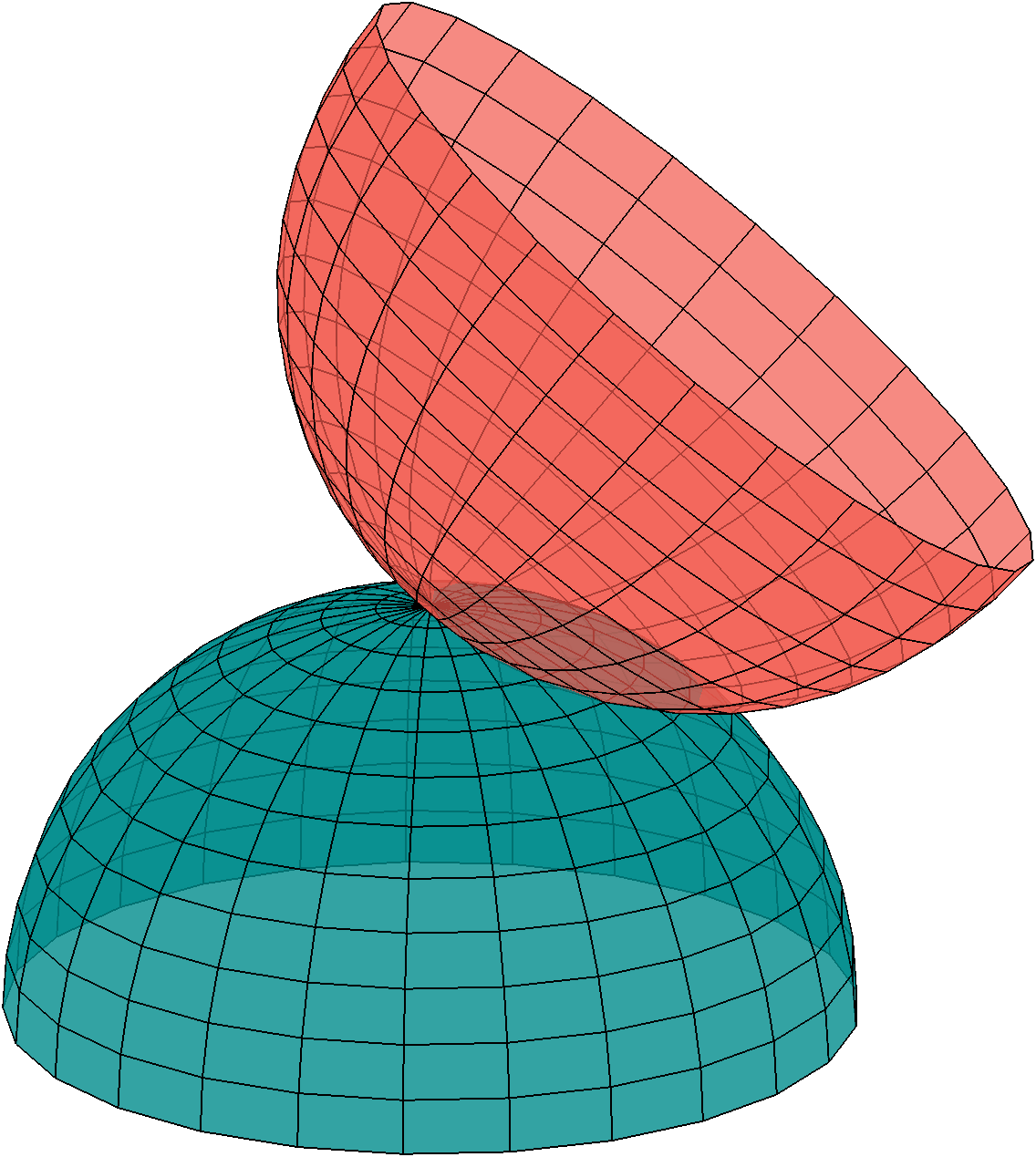}
  \caption{Illustration of the common circles as intersections of the Ewald spheres. 
    Left: Two hemispheres $\mathcal H_0$ (green) and $\mathcal H_t$ (red) intersect in a common circle. 
  	The north pole of both hemispheres is at $\zb0$. 
    Right: For real-valued (lossless) interactions, a  dual common circle at the intersection of $\mathcal H_0$ (green) and $-\mathcal H_t$ (red) exists. Data must agree on the two intersections.
     \label{fig:intersection_spheres}
}
\end{center}
\end{figure}
The \emph{scaled squared energy} 
\begin{equation} \label{eq:nu}
  \nu_t \colon\mathcal B^2_{k_0} \to[0,\infty),\qquad
 \nu_t(\bk) \coloneqq \frac2\pi \kappa^2(\bk)\abs{\ktran_2[m_t](\bk)}^2
\end{equation}
depends only on the measurements $m_t$, $t \in [0,T]$,
and is related to the scattering potential~$f$ via
\begin{equation}\label{eq:nu-F}
\nu_t(\bk)
= \abs{\ktran_3[f] (R_t\bh(\bk))}^2
,\qquad \bk \in \mathcal B^2_{k_0}.
\end{equation}
For every $t$, we see that $\nu_t$ provides the Fourier transform of $f$ on the set
\[ 
\mathcal H_t \coloneqq\left\{ R_t\bh(\bk):\bk \in \mathcal B^2_{k_0} \right\}
= \{R_t\bxh:\|\bxh+k_0\be^3\|=k_0,\,y_3>-k_0\}
,\quad t\in[0,T],
\]
which is a hemisphere with radius~$k_0$ and center $-k_0R_t\be^3$.
For two time steps $s,t$, 
the intersection $\mathcal H_s\cap\mathcal H_t$, $s\ne t$ is an arc of a circle, see \autoref{fig:intersection_spheres} left.
The idea is to parameterize the common circle by pairs $\bk_{s,t}$, $\bk_{t,s}$ such that 
$ R_s \bh(\bk_{s,t}) = R_t \bh(\bk_{t,s})\in\mathcal H_s\cap\mathcal H_t$,
then we have
$$\nu_s(\bk_{s,t}) = \nu_t(\bk_{t,s}).$$

If $f$ is real-valued, there is no absorption, and its Fourier transform is conjugate symmetric,
i.e.,
$\ktran_3[f](\bxh)=\overline{\ktran_3[f](-\bxh)}$
for all $\bxh\in\R^3$.
Then there is another pair of \emph{dual common circles} in the intersection $\mathcal H_s\cap (-\mathcal H_t)$,
where the measurements must agree,
see \autoref{fig:intersection_spheres} right.

The next theorem states that the incremental rotation $R_s^\top R_t$ can be determined from finding the common circle pairs in the data $\nu_s$, $\nu_t$.
We represent a rotation $Q\in\SO$
via Euler angles,
\begin{equation} \label{eq:euler} 
Q = Q^{(3)}(\varphi)\,Q^{(2)}(\theta)\,Q^{(3)}(\psi), 
\quad 
\varphi,\psi\in[0,2\pi),\ \theta\in[0,\pi],
\end{equation}
where 
\[ 
Q^{(2)}(\alpha) \coloneqq \begin{pmatrix}\cos\alpha&0&\sin\alpha\\0&1&0\\-\sin\alpha&0&\cos\alpha\end{pmatrix}
\quad\text{and}\quad 
Q^{(3)}(\alpha) \coloneqq \begin{pmatrix}\cos\alpha&-\sin\alpha&0\\\sin\alpha&\cos\alpha&0\\0&0&1\end{pmatrix}
\]
are rotations around the y and z axis, respectively.

\begin{theorem}[{\cite[Thm 3.6]{ElbQueSchSte23}}] \label{th:commonCircle}
Let $s,t\in[0,T]$. 
Assume that there exist unique angles $\varphi,\psi \in [0,2\pi)$ 
and $\theta\in [0,\pi]$ 
such that
\begin{alignat}{5}
  \nu_s(\bgam^{\varphi,\theta}(\beta)) 
  &= 
  \nu_t(\bgam^{\pi-\psi,\theta}(-\beta))
  \quad&&\forall\, \beta\in [-\tfrac\pi2,\tfrac\pi2]\quad\text{and}
  \label{eq:nuGamma}
  \\
  \nu_s(\bgam^{*,\varphi,\theta}(\beta)) 
  &= 
  \nu_t(\bgam^{*,\pi-\psi,\theta}(\beta))
  \quad&&\forall\, \beta\in [-\tfrac\pi2,\tfrac\pi2],
  \label{eq:nuGammaDual}
\end{alignat}
where
\begin{align}\label{eq:gammaEuler}
\bgam^{\varphi,\theta}(\beta) 
&\coloneqq \hphantom{-}\frac{k_0}2\sin\theta\cdot(\cos\beta-1)\begin{pmatrix}\cos\varphi\\
\sin\varphi\end{pmatrix}+k_0\cos\tfrac\theta2\,\sin\beta\begin{pmatrix}-\sin\varphi\\\cos\varphi\end{pmatrix}
\quad\text{and}
\\
\label{eq:dualEllipseEuler}
\zb{\gamma}^{*,\varphi,\theta}(\beta)
&\coloneqq -\frac{k_0}2\sin\theta\cdot(\cos\beta-1)\begin{pmatrix}\cos\varphi\\\sin\varphi\end{pmatrix} -k_0\sin\tfrac\theta2\,\sin\beta\begin{pmatrix}-\sin\varphi\\\cos\varphi\end{pmatrix}.
\end{align}
Then the incremental rotation is
\begin{equation}\label{eq:commonCircle}
  R_s^\top R_t = Q^{(3)}(\varphi) Q^{(2)}(\theta) Q^{(3)}(\psi).
\end{equation}
\end{theorem}

The elliptic arcs \eqref{eq:gammaEuler} and \eqref{eq:dualEllipseEuler} provide parameterizations of the
common circle $\mathcal H_s\cap\mathcal H_t$ and the dual common circle $\mathcal H_s\cap(-\mathcal H_t)$, respectively.
In general, it is not known when the uniqueness assumption in the last theorem is fulfilled.
However, it was shown that under certain assumptions on $f$ and if the measurements cover the whole $\SO$,
the object is uniquely determined \cite{KurZic21}.

\subsection{Infinitesimal Common Circle Method} \label{sec:cont-cc}
The infinitesimal common circle method relies on comparing derivatives of $\nu_t$.
We assume that the rotation $R_t$ is continuously differentiable in time $t\in[0,T]$. 
The \emph{angular velocity} corresponding to $R_t$ is the unique vector 
$\zb\omega_t\in\R^3$ satisfying
\begin{equation}\label{eq:omega}
  R_t^\top R_t'\zb y
  = \zb\omega_t\times \zb y
  \qquad \forall\,
	t\in[0,T],\;\zb y\in\R^3.
\end{equation}
We express $\zb\omega_t$ in cylindrical coordinates
\begin{equation}\label{eq:omega-decomp}
  \zb\omega_t 
  = ( \rho_t\cos\phi_{t}, \, \rho_t\sin\phi_{t}, \, \zeta_t )^\top
  ,\qquad 
  \phi_t\in[0,\pi),\quad 
  \rho_t,\,
  \zeta_t\in\R.
\end{equation}

\begin{theorem}[{\cite[Thm 4.2]{ElbQueSchSte23}}] \label{th:infRecon}
Let the rotation matrix $R\in C^1([0,T]\to\SO)$ be continuously differentiable and $t\in[0,T]$.
Let further $\phi\in[0,\pi)$ be a unique angle
with the property that there exist $\rho,\zeta\in\R$ such that
\begin{equation}\label{eq:infRecon}
\partial_t\nu_t\left(r\begin{psmallmatrix}\cos\phi\\\sin\phi\end{psmallmatrix}\right)
=\left(\rho\left(k_0-\sqrt{k_0^2-r^2}\right)+r\zeta\right)
\left<\nabla\nu_t\left(r\begin{psmallmatrix}\cos\phi\\\sin\phi\end{psmallmatrix}\right),
\begin{psmallmatrix}
-\sin\phi\\
\cos\phi
\end{psmallmatrix}\right>
\quad\forall\; r\in(-k_0,k_0).
\end{equation}
If there exist at least two values
$
r\in(-k_0,k_0)\setminus\{0\}
$
such that
$\left<\nabla\nu_t\left(r\begin{psmallmatrix}\cos\phi\\\sin\phi\end{psmallmatrix}\right),\begin{psmallmatrix}-\sin\phi\\\cos\phi\end{psmallmatrix}\right>\ne0
$,
then the angular velocity \eqref{eq:omega-decomp} is given by $\zb\omega_t = (\rho\cos\phi,\rho\sin\phi,\zeta)^\top$.
\end{theorem}

Recently, it was shown \cite{ElbSchm25} that \eqref{eq:infRecon} always has a unique solution if the object $f$ obeys a certain asymmetry condition.
 
Given the reconstructed angular velocity $\zb\omega_t$ for all $t$ and an initial rotation $R_0\in\SO$, 
the rotation $R_t$ is 
the unique solution of the linear initial value problem
\begin{equation} \label{eq:ode}
    R_t' = R_tW_t,\quad t\in[0,T], 
\end{equation}
see \cite[Thm 4.3]{ElbQueSchSte23},
with the skew-symmetric matrix
\begin{equation}\label{eq:W}
  W_t = 
  \begin{pmatrix}
    0&-\omega_{t,3}&\omega_{t,2}\\ 	
    \omega_{t,3}&0&-\omega_{t,1}\\ 
    -\omega_{t,2}&\omega_{t,1}&0
  \end{pmatrix}
  ,\qquad 
  \zb\omega_t = (\omega_{t,1},\omega_{t,2},\omega_{t,3})^\top.
\end{equation}

\section{Reconstruction of the Rotation} \label{sec:recon}

In this section, we describe numerical algorithms to reconstruct the rotations, assuming that $\nu_t$ from \eqref{eq:nu} are given.
We start with the infinitesimal common circle method, 
serving as initial guess for the direct method.

\subsection{Infinitesimal Common Circle Method}
\label{sec:recInf}

\subsubsection{Estimate of the Angular Velocity}

Let $t \in (0,T)$.
We reconstruct the angular velocity $\zb\omega_t=(\rho_t \cos\phi_t,\rho_t\sin\phi_t,\zeta_t)$ using a variational approach for \autoref{th:infRecon}.
Following \cite{ElbQueSchSte23}, 
we set for $r \in (-k_0,k_0)$ and $\phi\in[0,\pi)$,
\begin{equation}\label{eq:gh}
  \begin{aligned}
    g_{\phi}(r)
    &\coloneqq \partial_t\nu_t(r\cos(\phi),r\sin(\phi)),
    \\
    p_{\phi}(r)
    &\coloneqq
    \frac{1}{r} (k_0-\sqrt{k_0^2-r^2})\, \partial_\phi\nu_t(r\cos(\phi),r\sin(\phi))
    ,\\
    q_{\phi}(r)
    &\coloneqq
    \partial_\phi\nu_t(r\cos(\phi),r\sin(\phi)),
  \end{aligned}
\end{equation}
which depend only on the data~$\nu_t$.
Then \eqref{eq:infRecon} becomes
\begin{equation} \label{eq:d-nu-g}
  g_{\phi_t}(r) 
  = \rho_t\, p_{\phi_t}(r) + \zeta_t\, q_{\phi_t}(r)
  ,\qquad r\in(-k_0,k_0).
\end{equation}
We use a least squares approach to solve \eqref{eq:d-nu-g} for $\rho_t$, $\phi_t$ and $\zeta_t$.
We minimize 
the functional
\begin{equation} \label{eq:J}
  \mathcal J_t(\rho,\phi,\zeta)
  \coloneqq \norm{g_{\phi} - \rho \, p_{\phi} - \zeta \, q_{\phi}}_{L^2(-k_0,k_0)}^2
  ,\qquad \rho,\zeta\in\R,\;\phi\in[0,\pi),
\end{equation}
as follows.
For every $\phi \in [0,\pi)$ on a fixed grid, we compute the minimizer of 
$\mathcal J(\cdot,\phi,\cdot)\colon\R^2\to\R$
given by
\begin{equation} \label{eq:Jmin}
  \begin{pmatrix} \hat\rho(\phi)\\ \hat\zeta(\phi) \end{pmatrix} 
  = 
  \begin{pmatrix}
    \inn{p_{\phi},p_{\phi}} 
    & \inn{p_{\phi},q_{\phi}} \\ \inn{p_{\phi},q_{\phi}} & \inn{q_{\phi}, q_{\phi}}
  \end{pmatrix}^{-1}
  \begin{pmatrix} \inn{g_{\phi},p_{\phi}}\\ \inn{g_{\phi},q_{\phi}} \end{pmatrix},
\end{equation}
where $\left<\cdot,\cdot\right>$ is the inner product on $L^2(-k_0,k_0)$ and we assume that $p_{\phi}$ and $q_{\phi}$ are linearly independent.
For all such $\phi\in[0,\pi)$, we set 
$
  j_t(\phi)
  = 
  \mathcal J_t(\hat\rho(\phi),\phi,\hat\zeta(\phi)),
$
and then $\hat\phi=\arg\min_{\phi\in[0,\pi)} j_t(\phi)$.
Finally, we approximate the angular velocity 
$\zb\omega_t\approx (\hat\rho(\hat\phi)\, \cos\hat\phi,\,
    \hat\rho(\hat\phi)\, \sin\hat\phi,\,
    \hat\zeta(\hat\phi) )$.
The method is summarized in \autoref{alg:infOmega}, see \cite[Algo~2]{ElbQueSchSte23}.

\begin{algorithm}
  \KwIn{Scaled squared energy $\nu_t(r_n \cos \phi_\ell, r_n\sin\phi_\ell)$ from \eqref{eq:nu} on polar grid $r_n\in[0,k_0)$, $n=1,\dots,N$,
    and $\phi_\ell\in[0,\pi)$, $\ell=1,\dots,L$.}
  
  \For{$\ell=1,\dots,L$}{
    Compute $g_{\phi_\ell}(r_n),$ $p_{\phi_\ell}(r_n),$ and $q_{\phi_\ell}(r_n)$, $n=1,\dots,N$ by \eqref{eq:gh}\;
    
    Compute 
    $\hat\rho(\phi_\ell)$ and $\hat\zeta(\phi_\ell)$ by \eqref{eq:Jmin};
    
    Compute $j(\phi_\ell) \coloneqq 
    \sum_{n=1}^{N} |g_{\phi_\ell}(r_n) - \hat\rho(\hat\phi_\ell) \, p_{\phi_\ell}(r_n) 
      - \hat\zeta(\hat\phi_\ell) \, q_{\phi_\ell}(r_n)|^2
    $
  }
  Set $\hat\phi$ as minimizer of $j(\phi_\ell)$ over $\ell=1,\dots,L$\;
  
  \KwOut{Angular velocity $\zb\omega_t \approx (\hat\rho(\hat\phi)\, \cos\hat\phi,\,
    \hat\rho(\hat\phi)\, \sin\hat\phi,\,
    \hat\zeta(\hat\phi) )$.}
  
  \caption{Reconstruction of the angular velocity $\zb\omega_t$ with the infinitesimal method}
  \label{alg:infOmega}
\end{algorithm}

The directional derivative 
$
\partial_\phi
\nu_t(r\cos\phi,r\sin\phi)
$
in \eqref{eq:gh}
can be computed from the measurements $m_t$ via the following lemma.

\begin{lemma}
  Let $r\in(-k_0,k_0)$ and $\phi\in[0,\pi)$.
  Setting $M_t(\bx) \coloneqq -\i \bx m_t(\bx)$,
  we have
  \begin{equation} \label{eq:nu_grad}
    \partial_\phi 
    \nu_t (r\cos\phi,r\sin\phi)
    =
    \frac{4}\pi (k_0^2-r^2)\, \Re\left( \ktran_2[m_t](r\cos\phi, r\sin\phi)\, 
    \,\partial_\phi \ktran_2[M_t] (r\cos\phi, r\sin\phi)
    \right).
  \end{equation}
  
\end{lemma}

\begin{proof}
  Let $\bk\in \mathcal B_{k_0}^2$.
  We note that 
  $
  \partial_\phi \nu_t(r\cos\phi,r\sin\phi)
  =
  \left< \nabla \nu_t (r\cos\phi, r\sin\phi), \begin{psmallmatrix}
      -\sin\phi\\
      \cos\phi
    \end{psmallmatrix} \right>
  $.
  We have 
  $\nabla \kappa^2(\bk) = -2\bk $.
  By \eqref{eq:nu}, we see that
  \begin{equation*}
    \nabla \nu_t(\bk)
    =
    \tfrac2\pi \left(-2\bk \abs{\ktran_2[m_t](\bk)}^2 + \kappa^2(\bk) 2\Re\left( \ktran_2[m_t](\bk) \nabla \ktran_2[m_t](\bk) \right) \right).
  \end{equation*}
  The differentiation property of the Fourier transform yields
  $
    \nabla\ktran_2[m_t](\bk)
    =
    \ktran_2[M_t](\bk).
  $
  Inserting $\bk=(r\cos\phi,r\sin\phi)^\top$ and using that $\langle \bk,(-\sin\phi,\cos\phi)^\top\rangle=0$ and $\kappa^2(r\cos\phi,r\sin\phi) = (k_0^2 - r^2)$, we obtain the assertion.
\end{proof}

\subsubsection{Temporal Regularization of the Angular Velocity}

In our application, 
the object moves smoothly, 
so we assume that $R_t$ is twice differentiable in time,
and therefore the angular velocity $\zb\omega_t$ is differentiable.
We add to \eqref{eq:J} a regularization term penalizing the variation of $\zb\omega_t$.
Considering $\rho$, $\phi$ and $\zeta$ as functions in $t$,
we minimize the functional
\begin{equation} \label{eq:Jreg}
  \tilde{\mathcal J}_\lambda(\rho,\phi,\zeta)
  =
  \int_0^T \left(\mathcal J_t(\rho_t,\phi_t,\zeta_t) + \lambda \mathcal G_t(\rho,\phi,\zeta)\right) \,\mathrm d t
\end{equation}
with some regularization parameter $\lambda>0$
and 
\begin{align}
\mathcal G_t(\rho,\phi,\zeta)
&\coloneqq
\norm{\partial_t \zb\omega(\rho_t,\phi_t,\zeta_t)}^2
=
\norm{\partial_t\begin{psmallmatrix}\rho_t\cos\phi_t\\\rho_t\sin\phi_t\\\zeta_t\end{psmallmatrix}}^2
  =
  (\rho_t')^2 + |\rho_t| (\phi_t')^2+(\zeta_t')^2.
\end{align}
A straightforward calculation shows the following formulas for the derivatives of $\tilde{\mathcal J}_\lambda$.
\begin{lemma}
  Let $\lambda>0$, $\phi\in C^2([0,T]\to[0,\pi))$, and 
  $\rho,\zeta\in C^2([0,T]\to\R)$.
  The functional derivatives of $\tilde{\mathcal J}_\lambda$ are given by
  \begin{align}
  \delta_\rho \tilde{\mathcal J}_\lambda(\rho,\phi,\zeta)
  &=
  \int_0^T \left(\partial_\rho {\mathcal J}_t(\rho_t,\phi_t,\zeta_t) -2\lambda\rho_t'' + \lambda(\phi_t')^2 \right) \,\mathrm dt,
  \\
  \delta_\phi \tilde{\mathcal J}_\lambda(\rho,\phi,\zeta)
  &=
  \int_0^T \partial_\phi {\mathcal J}_t(\rho_t,\phi_t,\zeta_t) -2\lambda\rho_t\phi_t'' \,\mathrm dt,
  \\
  \delta_\zeta \tilde{\mathcal J}_\lambda(\rho,\phi,\zeta)
  &=
  \int_0^T \partial_\zeta {\mathcal J}_t(\rho_t,\phi_t,\zeta_t)
  -2\lambda\zeta_t'' \,\mathrm dt,
  \end{align}
  where
  \begin{align*}
  \partial_\rho \mathcal J_t(\rho,\phi,\zeta)
  &=
  -2 \int_{-k_0}^{k_0} p_\phi(r)
  \left( g_\phi(r)-\rho p_\phi(r)-\zeta q_\phi(r) \right) \dd r,
  \\
  \partial_\phi \mathcal J_t(\rho,\phi,\zeta)
  &=
  2\int_{-k_0}^{k_0}
  ( g_\phi(r)-\rho p_\phi(r)-\zeta q_\phi ) 
  \Big(  r\partial_\phi\partial_t - \Big(\frac{\rho }{r}({k_0^2-\sqrt{k_0^2-r^2}})+\zeta \Big) \partial_\phi^2 \Big) \nu_t(r\zb\phi) \dd r,
  \\
  \partial_\zeta \mathcal J_t(\rho,\phi,\zeta)
  &=
  -2\int_{-k_0}^{k_0}
  q_\phi(r)
  \left( g_\phi(r)-\rho p_\phi(r)-\zeta q_\phi(r) \right) \dd r.
  \end{align*}
\end{lemma}

We minimize \eqref{eq:Jreg} using gradient descent with step-width $\alpha>0$.
We assume to have a starting solution 
$\rho^{(0)}\in\R$, $\phi^{(0)}\in[0,\pi)$, and $\zeta^{(0)}\in\R$, 
e.g.\ from \autoref{alg:infOmega}, and set 
\begin{equation}\label{eq:gd}
\begin{split}
\rho^{(k+1)} 
=
\rho^{(k)} - \alpha \partial_\rho \mathcal J_t(\rho^{(k)},\phi^{(k)},\zeta^{(k)})
,\\
\phi^{(k+1)} 
=
\phi^{(k)} - \alpha \partial_\phi \mathcal J_t(\rho^{(k)},\phi^{(k)},\zeta^{(k)})
,\\
\zeta^{(k+1)} 
=
\phi^{(k)} - \alpha \partial_\zeta \mathcal J_t(\rho^{(k)},\phi^{(k)},\zeta^{(k)}).
\end{split}
\end{equation}

\subsubsection{Rotation Matrix}
Following \cite{ElbQueSchSte23},
we compute the rotation $R_t$ with given angular velocity $\zb\omega_t$
by solving the initial value problem \eqref{eq:ode} with the forward \emph{Euler method} on $\R^{3 \times 3}$
followed by a {projection} to $\SO$ in each step.
More precisely,
denote
the singular value decomposition (SVD) of a matrix $A\in\R^{3\times3}$ 
by $A = U\Sigma V^\top$,
and set $\text{Polar}(A) \coloneqq U V^\top$.
The orthogonal projection $\text{P}_{Q}$
of the tangent space  
$$
\mathcal T_{Q}\SO
\coloneqq
\{Q S: S \in \R^{3\times3},\, S=-S^\top 
\}
$$
of $\SO$ at $Q \in \SO$
onto $\SO$ 
with respect to the 
Frobenius norm
is given by
\begin{equation} \label{eq:Polar}
\text{P}_{Q} ( W) \coloneqq \text{Polar}(Q +  W)\in\SO
,\qquad Q\in\SO,\;  W\in\mathcal T_{Q}\SO,
\end{equation}
cf.\ \cite{Moa02}.
Using integer time steps $t=0,1,\dots,T$ and assuming that $R_0\in\SO$ is given,
we approximate the rotation $R_t$ by
\begin{equation} \label{eq:Euler-proj}
\begin{aligned}
  \mathbf R_0
  &\coloneqq R_0,
  \\
  \mathbf R_{{t+1}} &\coloneqq \text{P}_{\mathbf R_{t}}( \mathbf R_{t} W_{t}).
\end{aligned}
\end{equation}

\subsection{Regularized Direct Common Circle Method} 

The direct method solves \eqref{eq:nuGamma} and \eqref{eq:nuGammaDual} 
to find the Euler angles $\varphi\in [0,2\pi)$, $\theta\in [0,\pi]$, $\psi\in[0,2\pi)$ 
of the incremental rotation $R_s^\top R_t$. 
We minimize the least-squares functional 
\begin{equation}\label{eq:dist}
  \mathcal E_{s,t}(\varphi,\theta,\psi)
  \coloneqq
  \int_{-\pi/2}^{\pi/2} 
  \abs{\nu_t (\bgam^{\pi-\psi,\theta}(-\beta))
    -\nu_s (\bgam^{\varphi,\theta}(\beta))}^2
    + 
    \abs{\nu_t (\bgam^{*,\pi-\psi,\theta}(\beta))
      -\nu_s (\bgam^{*,\varphi,\theta}(\beta))}^2
  \dd \beta
\end{equation}
over $\varphi,\psi\in[0,2\pi)$ and $\theta\in[0,\pi]$, cf.\ \cite{ElbQueSchSte23}.
The true Euler angles \eqref{eq:euler} of $R_s^\top R_t$ constitute a global minimizer of $\mathcal E_{s,t}$.
However, there is no theoretical guarantee that $\mathcal E_{s,t}$ has only one global minimum.
For noisy measurements,
$\mathcal E_{s,t}$ might have a global minimum not corresponding to the true rotation.
Furthermore, 
the optimization becomes quite hard
as $\mathcal E_{s,t}$ is non-convex and there may exist many local minima.

Contrary to \cite{ElbQueSchSte23},
we assume that we already have an estimate $\tilde R_t$ of $R_t$,
and therefore $R_s^\top R_t$ is in the proximity of $\tilde R_s\tilde R_t$.
This estimate can come e.g.\ from the infinitesimal method of \autoref{sec:recInf}. 
We denote the angle of a rotation $Q\in\SO$ by 
\begin{equation} \label{eq:angle}
\angle(Q)=\arccos((\operatorname{trace}(Q)-1)/2)\in[0,\pi],
\end{equation}
which induces a distance of two rotations on $\SO$ via 
$$
d(R,Q) = \angle(R^\top Q)
,\qquad R,Q\in\SO.
$$
Then, we add to \eqref{eq:dist} as regularization term the distance to the estimated rotation $\tilde R_s^\top \tilde R_t$, i.e.,
\begin{equation} \label{eq:direct-reg}
  \tilde{\mathcal E}_{s,t}^\lambda (\varphi,\theta,\psi)
  \coloneqq
  \mathcal E_{s,t}(\varphi,\theta,\psi)
  + \lambda\, d\big(\tilde R_s^\top \tilde R_t, Q^{(3)}(\varphi) Q^{(2)}(\theta) Q^{(3)}(\psi)\big),
\end{equation}
with a regularization parameter $\lambda>0$.
Then $\tilde{\mathcal E}_{s,t}$ can be minimized via a standard algorithm like gradient descent.
The integral \eqref{eq:dist} can be discretized via quadrature,
and $\nu_t$ needs to be interpolated if it is given on a grid.

\section{Numerics}
\label{sec:numerics}

For $N\in\N$, define the grid 
$$\cI_N \coloneqq \{-\tfrac N2, \dots, \tfrac N2-1\},$$
and denote the pixel size by $p>0$.
We have a video of the (complex-valued) total field 
$$
\utot_t(\bx)
=
\ui(\bx) + u_t(\bx),
\qquad \forall \bx\in p\cI_N^2,\ t=0,\dots,T-1,
$$
i.e.\ the sum of incidence field $\ui$ from \eqref{eq:plane_wave} and scattered field $u_t$.
The time index $t=0,\dots,T-1$ corresponds to the $t$-th frame of the video.
It is convenient to decompose
$$
\utot_t(\bx)
=
{a_t(\bx)}\, \e^{\i \varphi_t(\bx)}
$$
into its amplitude $a_t(\bx)=|\utot_t(\bx)|$ and phase $\varphi_t (\bx)$.
For the latter, a phase unwrapping needs to be performed in general, cf. \cite{CheSta98,GhiPri98},
but is not necessary in our examples.
We emphasize that the complex field is measured and therefore no phase retrieval is required as opposed to \cite{BeiQue23,Mal93}.

\subsection{Data Preprocessing}

\textbf{Renormalization and estimation of the incidence field:} By \eqref{eq:plane_wave}, the incident field $\ui$ at the measurement plane is $\e^{\i k_0 \rM}$.
To compute the scattered field $u=\utot-\ui$, we assume the incident field $\ui$ to be constant over the whole measurement plane, but it might vary slightly over time $t$ due to data acquisition effects.
We approximate the incident field by $\ui_t \approx \e^{\i \varphi_t^{\mathrm{med}}}$,
where
$
\varphi_t^{\mathrm{med}}
$
is the median of the phase 
$
\{ \varphi_t(\bx) :
\bx\in p\cI_N^2 \}$.
The median over the whole recorded image should be unaffected by the values inside the object.
Theoretically, one would need to consider $\varphi$ on the $2\pi$-periodic torus, but this is not necessary here as the phase only varies moderately.
We have 
$$
u_t(\bx)
=
\utot_t(\bx)-\ui_t
\approx
{a_t(\bx)}\, \e^{\i \varphi_t(\bx)} - \e^{\i \varphi_t^{\mathrm{med}}}.
$$

\textbf{Rytov approximation:} 
The Rytov approximation assumes that the phase gradient is small and therefore often yields more accurate results than the Born approximation, which assumes the total phase change to be small,
cf.\ \cite{KakSla01,MueSchuGuc15_report}.
Mathematically, it corresponds to a transformation of the measurement data,
then the reconstruction can be done exactly as with the Born approximation used in the theoretical derivations, see \cite{FauKirQueSchSet23}.
Setting $a_t^\mathrm{med}$ as the median of the amplitude $a_t$,
we use as our Rytov data
\begin{equation*}
  u_t^\mathrm{Rytov}(\bx) 
  = 
  \ui_t\cdot \left(\i (\varphi_t(\bx) - \varphi_t^\mathrm{med}) + \log\left(\frac{a_t(\bx)} {a_t^\mathrm{med}}\right)\right).
\end{equation*}

\textbf{Soft cutoff:} 
We first cut out the images to a square so that the object is approximately in the center. 
As the outer part of the images contains mainly noise and no object information,
we apply the soft, circular cutoff $C^1$ function 
\begin{equation} \label{eq:cutoff}
  c_{r_1,r_2}(\bx)=
  \begin{cases}
    1, & |\bx|\le r_1,\\
    \frac{(r_2-|\bx|)^2 (2|\bx|+r_2-3r_1)} {(r_2-r_1)^3}, & r_1<|\bx|<r_2,\\
    0, & |\bx|\ge r_2,
  \end{cases}
\end{equation}
with parameters $r_2>r_1>0$.
We set our preprocessed, transformed measurements 
$$
m_t(\bx)\coloneqq c_{r_1,r_2}(\bx)u_t^\mathrm{Rytov}(\bx).
$$

\textbf{Pre-smoothing:}
To mitigate high-frequency noise, we apply a 3D Gaussian smoothing with standard deviation $0.65$ pixels to the transformed data~$m$.

\subsection{Reconstruction Steps}
\label{sec:reconstruction-steps}

We estimate the rotation by the infinitesimal and the direct method.
Then we can use this to reconstruct the object.

\textbf{Infinitesimal common circle method:}
We first apply the infinitesimal method of \autoref{alg:infOmega} to obtain the angular velocity $\zb\omega_t$.
This method requires to evaluate the derivatives of $\nu_t(\bk) =\frac2\pi \kappa^2(\bk) |\ktran_2[m_t](\bk)|$, see \eqref{eq:nu}, i.e., the Fourier transform of $m_t$, on a polar grid $\bk=(r\cos\phi,r\sin\phi)$. 
We approximate the Fourier transform on a polar grid using the nonuniform fast Fourier transform (NFFT) \cite{PlPoStTa18} of~$m_t$, implemented in \cite{nfft3}.
Then we compute the partial derivatives of $\nu_t$ via finite differences with the 3D analogue of the Sobel filter $\frac18\begin{psmallmatrix}-1&-2&-1\\0&0&0\\1&2&1\end{psmallmatrix}$.

We improve this first estimate of the angular velocity $\zb\omega_t$ by minimizing~\eqref{eq:Jreg}
using the gradient descent \eqref{eq:gd} 
with 50 iterations, $\lambda=0.1$, and $\alpha=10^{-10}$.
From the resulting new estimate of $\zb\omega_t$,
we obtain the rotations $R_t$ by 
\eqref{eq:Euler-proj}.

\textbf{Direct common circle method:}
We minimize the regularized functional \eqref{eq:direct-reg}.
As starting solution $\tilde R_t$, 
we use our reconstructed rotations from the infinitesimal method.
The integral in \eqref{eq:dist} is approximated by 200 equidistant points $\beta$, 
and the function $\nu_t$ is evaluated on arbitrary points via a cubic spline interpolation.
We minimize \eqref{eq:direct-reg} with the Nelder–Mead
downhill simplex method implemented in Matlab’s \texttt{fminsearch}.
We also tested a gradient-based optimizer, which usually gave similar results and computation times.

We go three times over all time steps. 
The direct method encounters instabilities when the rotations $R_t$ and $R_s$ are too similar or differ by nearly 180\,°,
because the measured hemispheres are close together.
Therefore, we choose $t-s$ in a range such that neither occurs.
As the motion is approximately periodical, 
we obtain a rough estimate of the number of frames for a full rotation by maximizing the correlation between the images $m_t$.
This way we have an estimate for which $|s-t|$ we can expect $R_t\approx R_s$.
Because the direct method is instable for small rotations, 
we do not use it for the first few time steps~$t$. 
Instead, we linearly interpolate (in the rotation angle) between the rotations.
As final postprocessing, we apply a mean filter with respect to time $t$ to the reconstructed rotations~$R_t$.

\textbf{Translations:}
Object translations, see \eqref{eq:f_t}, do not affect the reconstruction of the rotations, because $\nu_t$ is translation-invariant;
they only affect the object reconstruction.
In many subsequent examples, the translational shifts are small enough to be negligible.
We estimate such shifts by the following circle fitting approach.
We first preprocess $|u_t|$ with a median filter, and perform a circle fitting (in Matlab's imaging toolbox) on regions of the image where $|u_t|$ exceeds a threshold. 
The estimated shift is determined by the center of the fitted circle or ellipse. 
Finally, we interpolate the image using these estimated shifts.

\textbf{Object reconstruction:}
Having estimated the rotations $R_t$ in the previous step and the measurements $m_t$, 
we reconstruct the scattering potential~$f$ 
using the inverse NDFT {(nonuniform discrete Fourier transform)} method with 12 conjugate gradient iterations, as described in \cite[sect~5.2]{BeiQue22} and implemented in Matlab.\footnote{Matlab toolbox \url{https://github.com/michaelquellmalz/FourierODT}}
Then, by \eqref{eq:f}, we calculate from $f$ the refractive index 
$$
n(\bx)
=
n_0
\sqrt{\frac{f(\bx)}{k_0^2}+1}.
$$

\subsection{In Silico Phantom with Beam Propagation Method}

We simulate a phantom refractive index $n(\bx)$ consisting of a large ellipsoid with a major axis of \SI{6}{\micro\meter} and constant refractive index containing many small balls with higher refractive index, see \autoref{fig:bpm-n}.
This object has similar parameters and shape as the Neuroblastoma cell we will consider in \autoref{sec:nb} below. 
In order to avoid the so-called ``inverse crime'',
we simulate the measurements $\utot_t$ with the beam propagation method (BPM) described in \cite{Kam16}, see also \cite{ChoFanOhBadDas07,VanRoey81BPM}. 
This is a non-linear model of wave propagation \eqref{eq:totII} that is more accurate than the (linear) Born and Rytov approximations on which the common circle method is based.
We use $T=200$ frames and a full rotation around the $x_2$ axis with constant angular velocity.
The measurements of $\utot$ are simulated on $N\times N$ pixels with $N=420$ and the refractive index $n$ is discretized on $192\times192\times192$ pixels.

We assess the accuracy of the reconstructed rotation $R_t^\mathrm{rec}$ to the ground truth $R_t$ as the angle $\angle (R_t^\top R_t^\mathrm{rec})$, see \eqref{eq:angle}, averaged over all time steps $t$.

We first perform the common circles method with exact data from the BPM.
Here, the infinitesimal method already gives quite good results,
the rotations $R_t$ are reconstructed with an error of 6.8\,° averaged over time $t$, see \autoref{fig:bpm-error-exact}.
The direct common circles method gives only a minor improvement.

In order to simulate a realistic noise, 
we extract from measured real-world data (see below) a rectangular $N\times N\times T$ part of the video, where no object lies and thus the measurements should be constant in theory.
Then we add this video to the simulated data~$\utot$, and apply the preprocessing to obtain $m_t$, which is depicted in \autoref{fig:bpm-u}.
The infinitesimal common circles method gives approximately the right direction of the angular velocity $\zb\omega_t$, but significantly underestimates its norm $\|\zb\omega_t\|$, which is the speed of rotation, yielding unsatisfactory results.
This might be explained by the fact that the data consists of a superposition of the moving object and the noise, which has no rotational structure.
Therefore, we apply the direct common circles method running twice through all frames.
In particular, we minimize \eqref{eq:direct-reg} with $\lambda=60$, increase $s=0,10,20,\dots$, and $t=s+20,\dots,s+60$.
This yields to a small error of 4.2\,° averaged over all time steps.
The error of the rotations reconstructed with the common circle method for both exact and noisy data for each frame is shown in \autoref{fig:bpm-error}.
The reconstructed object is depicted in \autoref{fig:bpm-n-rec}.

\tikzset{font=\tiny}
\newcommand{\modelwidth} {4cm}
\pgfplotsset{
  colormap={parula}{
    rgb255=(53,42,135)
    rgb255=(15,92,221)
    rgb255=(18,125,216)
    rgb255=(7,156,207)
    rgb255=(21,177,180)
    rgb255=(89,189,140)
    rgb255=(165,190,107)
    rgb255=(225,185,82)
    rgb255=(252,206,46)
    rgb255=(249,251,14)}}
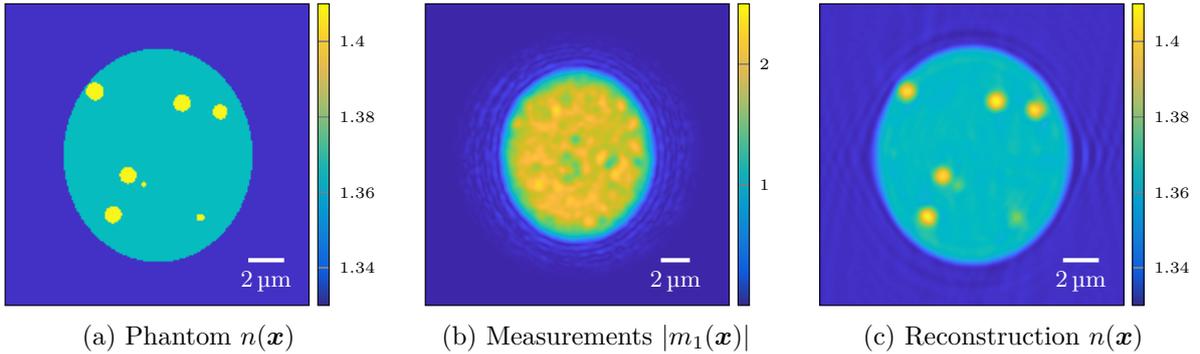
\begin{figure}[!ht]\centering
  \pgfmathsetmacro{\xmin}{-8.485281}
  \pgfmathsetmacro{\xmax}{8.485281}
  \begin{subfigure}{.33\textwidth}\centering
  \begin{tikzpicture}
    \begin{axis}[
      width=\modelwidth, height=\modelwidth,
      enlargelimits=false,
      scale only axis,
      axis on top,
      point meta min=1.33,point meta max=1.41,
      colorbar,colorbar style={
        width=.15cm, xshift=-0.5em,  
      },
      ticks = none,
      ]
      \addplot graphics [
      xmin=\xmin, xmax=\xmax,  ymin=\xmin, ymax=\xmax,
      ] {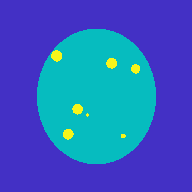};
      \path [draw, ultra thick, yscale=1, yshift=-14mm,xshift=12mm, white] (0,0) -- (2,0) node[below, midway] {\footnotesize\SI{2}{\micro\meter}};
    \end{axis}
  \end{tikzpicture}
  \caption{Phantom $n(\bx)$ \label{fig:bpm-n}}
  \end{subfigure}\hfill
  \begin{subfigure}{.33\textwidth}\centering
  \begin{tikzpicture}
    \begin{axis}[
      width=\modelwidth, height=\modelwidth,
      enlargelimits=false,
      scale only axis,
      axis on top,
      point meta min=0,point meta max=2.5,
      colorbar,colorbar style={
        width=.15cm, xshift=-0.5em,  ytick={1,2},
      },
      ticks = none,
      ]
      \addplot graphics [
      xmin=-10.5, xmax=10.5,  ymin=-10.5, ymax=10.5,
      ] {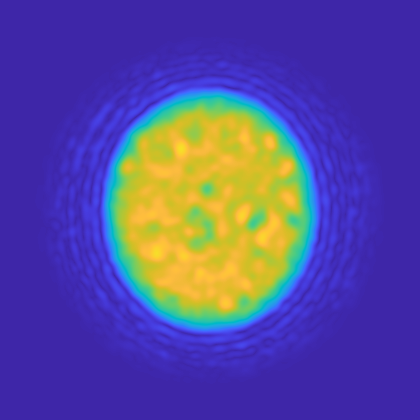};
      \path [draw, ultra thick, yscale=1, yshift=-14mm,xshift=11mm, white] (0,0) -- (2,0) node[below, midway] {\footnotesize\SI{2}{\micro\meter}};
    \end{axis}
  \end{tikzpicture}
  \caption{Measurements $|m_1(\bx)|$ \label{fig:bpm-u}}
  \end{subfigure}\hfill
  \begin{subfigure}{.33\textwidth}\centering
  \begin{tikzpicture}
    \begin{axis}[
      width=\modelwidth, height=\modelwidth,
      enlargelimits=false,
      scale only axis,
      axis on top,
      point meta min=1.33,point meta max=1.41,
      colorbar,colorbar style={
        width=.15cm, xshift=-0.5em,  
      },
      ticks = none,
      ]
      \addplot graphics [
      xmin=\xmin, xmax=\xmax,  ymin=\xmin, ymax=\xmax,
      ] {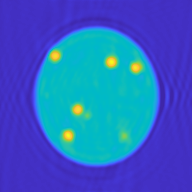};
      \path [draw, ultra thick, yscale=1, yshift=-14mm,xshift=12mm, white] (0,0) -- (2,0) node[below, midway] {\footnotesize\SI{2}{\micro\meter}};
    \end{axis}
  \end{tikzpicture}
  \caption{Reconstruction $n(\bx)$ 
  \label{fig:bpm-n-rec}}
  \end{subfigure}
  \caption{(a) Ground truth refractive index $n(\bx)$ of the phantom. 
  (b) Absolute value of the preprocessed data $|m_t(\bx)|$ for $t=1$ simulated with the BPM with added noise. 
  (c) Reconstruction of the refractive index $n(\bx)$ for noisy simulated data using the Rytov approximation and the rotations reconstructed via the common circle method. \label{fig:bpm}}
\end{figure}

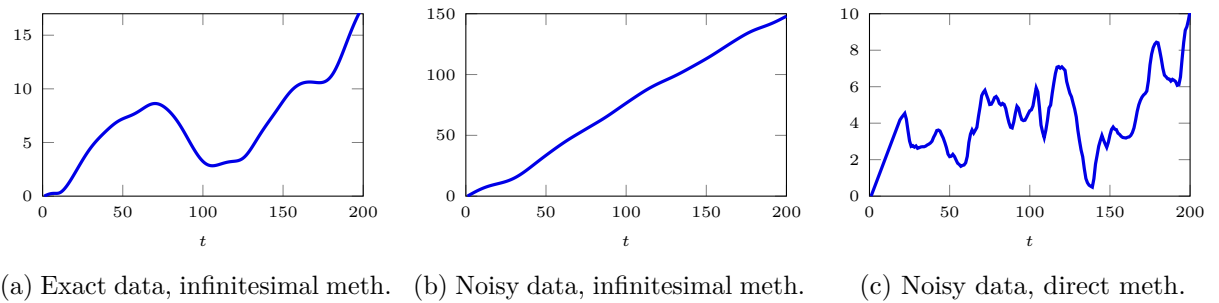
\begin{figure}[!ht] \centering
\begin{subfigure}{.32\textwidth}
  \begin{tikzpicture}
  \begin{axis}[xlabel=$t$, width=5.8cm, height=4cm, xmin=0, xmax=200, ymin=0, ymax=17,
    yticklabel style={
      /pgf/number format/fixed,
      /pgf/number format/precision=4
    },
    legend cell align={left},
    scaled y ticks=false
    ]
    \addplot+[blue!90!black, mark=none, very thick,
    ] table[x index=0,y index=1] {images/bpm-noise0_error.txt};
  \end{axis}
  \end{tikzpicture}
  \caption{Exact data, infinitesimal meth.\label{fig:bpm-error-exact}}
\end{subfigure}\hfill
\begin{subfigure}{.32\textwidth}
  \begin{tikzpicture}
  \begin{axis}[xlabel=$t$, width=5.8cm, height=4cm, xmin=0, xmax=200, ymin=0, ymax=150,
    yticklabel style={
      /pgf/number format/fixed,
      /pgf/number format/precision=4
    },
    legend cell align={left},
    scaled y ticks=false
    ]
    \addplot+[blue!90!black, mark=none, very thick,
    ] table[x index=0,y index=1] {images/bpm-noise1_error.txt};
  \end{axis}
  \end{tikzpicture}
  \caption{Noisy data, infinitesimal meth.}
\end{subfigure}\hfill
\begin{subfigure}{.32\textwidth}
  \begin{tikzpicture}
  \begin{axis}[xlabel=$t$, width=5.8cm, height=4cm, xmin=0, xmax=200, ymin=0, ymax=10,
    yticklabel style={
      /pgf/number format/fixed,
      /pgf/number format/precision=4
    },
    legend cell align={left},
    scaled y ticks=false
    ]
    \addplot+[blue!90!black, mark=none, very thick,
    ] table[x index=0,y index=4] {images/bpm-noise1_error.txt};
  \end{axis}
  \end{tikzpicture}
  \caption{Noisy data, direct meth.}
\end{subfigure}
  \caption{Error of the reconstructed rotation (in degree) at each frame $t$ for the neuroblastoma phantom of \autoref{fig:bpm} with exact (a) or noisy data (b), (c).
  \label{fig:bpm-error}}
\end{figure}

\goodbreak

\subsection{Real-World Data}

\subsubsection{Measurement Setup}

The incident wave is generated by a red laser with wavelength $\lambda_0 = \SI{640}{\nano\meter}$.
The surrounding medium is water with the refractive index $n_0 = 1.33$.
We assume that the image is focused at the center of the object, so we set $\rM=0$.
The measurement setup is described in detail in \cite{Mos25}.
The object is trapped and rotated in the acoustofluidic chamber from \cite{KvaMosThaRit22}, see also \cite{KvaPreRit21,KvaPreRit20}.
The complex-valued data is acquired via an interferometer adapted from \cite{GirSha13}.

\subsubsection{Reference Method for Comparison}

The optimization approach from \cite{Mos25} uses a more accurate forward model similar to the beam propagation method. 
It optimizes simultaneously the object function $f$, for which it uses a TV penalty, and the motion parameters.
However, it requires a somewhat decent initial solution in order to not get stuck in the wrong local minimum.
We consider the reconstructions as ``ground truth`` for the rotations.
Note that \cite{Mos25} uses a different coordinate system, where the light propagates in the first coordinate.
Therefore, we express the rotation $R$ in our coordinate system via the transformation 
$$
R = Q_0 \tilde R^\top Q_0^\top
, 
\qquad
Q_0
=
\begin{pmatrix}
0&1&0\\ 0&0&1\\ 1&0&0
\end{pmatrix},
$$
where $\tilde R$ is the rotation matrix from \cite{Mos25}.

\subsubsection{Three-Bead Object}

The imaged object consists of three spherical beads, each having a constant refractive index.
We have $T=200$ frames of $60\times60$ pixels with a resolution of $p=\SI{0.2}{\micro\meter}$.
The absolute value of the preprocessed data is visualized in \autoref{fig:beads3}.
A full rotation takes approximately 46 frames.

\begin{figure}\centering
  \pgfmathsetmacro{\xmin}{-6}
  \pgfmathsetmacro{\xmax}{6}
  \foreach \x in {1,6,11}{
  \begin{tikzpicture}
    \begin{axis}[
      width=\modelwidth, height=\modelwidth,
      enlargelimits=false,
      scale only axis,
      axis on top,
      point meta min=0,point meta max=6,
      colorbar,colorbar style={
        width=.15cm, xshift=-0.6em, 
      },
      ticks = none,
      ]
      \addplot graphics [
      xmin=\xmin, xmax=\xmax,  ymin=\xmin, ymax=\xmax,
      ] {images/beads3_u_t\x.png};
      \path [draw, ultra thick, yscale=1, yshift=-14mm,xshift=11mm, white] (0,0) -- (2,0) node[below, midway] {\footnotesize\SI{2}{\micro\meter}};
    \end{axis}
  \end{tikzpicture}
  }
  \caption{Absolute value of the transformed and preprocessed measurements $|m_t(\bx)|$ for the three-beads object dataset 
  for time steps $t\in\{0,5,10\}$ (left to right).
  }
  \label{fig:beads3}
  \end{figure}
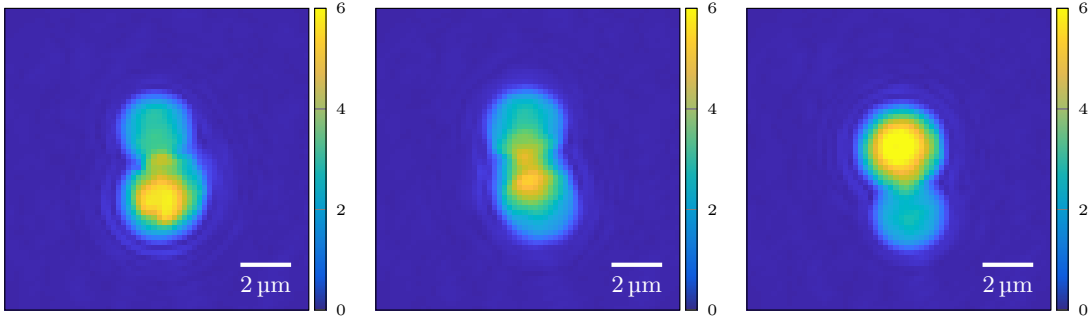

The reconstructed rotations are plotted in \autoref{fig:beads3-q}.
{For easy visualization, we represent a rotation~$R$ as quaternion $(q_0,q_1,q_2,q_3)$, which corresponds the rotation with angle $2\arccos(q_0)$ around the axis $(q_1,q_2,q_3)\in\R^3$ and is normalized so that $\sum_{i=0}^{3}q_i^2=1$.
}
The error to \cite{Mos25} is shown in \autoref{fig:beads3-error}. The reconstructed refractive index of the object is depicted in \autoref{fig:beads3-rec}.
We state the peak signal-to-noise ration (PSNR) and structural similarity index measure (SSIM) of the reconstructed scattering potential $f$ to the ground truth from the optimization approach~\cite{Mos25}.

\begin{figure}[ht]
  \centering
  \foreach \x in {1,2,3,4}{
    \begin{tikzpicture}
    \pgfmathtruncatemacro{\y}{\x + 4};
    \pgfmathtruncatemacro{\xmo}{\x - 1};
    \begin{axis}[xlabel=$t$, width=8cm, height=3.8cm, xmin=0, xmax=200, ymin=-1.05, ymax=1.05, 
      yticklabel style={
        /pgf/number format/fixed,
        /pgf/number format/precision=4
      },
      ytick={-1,0,1},
      legend pos={south east},
      legend cell align={left},
      scaled y ticks=false,
      ylabel={$q_\xmo$},
      ylabel style={rotate=-90},
      ]
      \addplot+[blue!90!black, mark=none, thick,
      ] table[x index=0,y index=\x] {images/beads3_quaternions.txt};
      \addplot+[red!90!black, mark=none, thick,
      ] table[x index=0,y index=\y] {images/beads3_quaternions.txt};
      \legend{Common circles,Optimization};
    \end{axis}
    \end{tikzpicture}\hfill
  }
  \caption{Reconstructed rotations for three-beads object. 
  The four plots are the components of the quaternions.  
  Blue: Common circle method. Red: Optimization approach \cite{Mos25}.}
  \label{fig:beads3-q}
\end{figure}
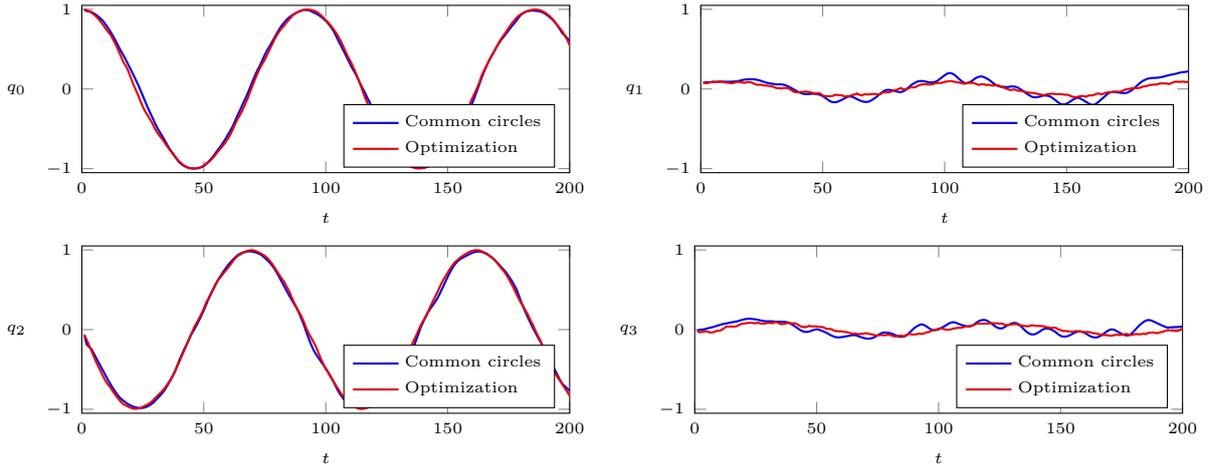

\begin{figure}[ht]
  \begin{tikzpicture}
    \begin{axis}[xlabel=$t$, width=8.6cm, height=5.0cm, xmin=0, xmax=200, ymin=0, ymax=185, 
      yticklabel style={
        /pgf/number format/fixed,
        /pgf/number format/precision=4
      },
      legend pos={south east},
      legend cell align={left},
      scaled y ticks=false
      ]
      \addplot+[blue!90!black, mark=none, very thick,
      ] table[x index=0,y index=2] {images/beads3_angle.txt};
      \addplot+[red!90!black, mark=none, thick,
      ] table[x index=0,y index=3] {images/beads3_angle.txt};
      \legend{Common circles,Optimization};
    \end{axis}
  \end{tikzpicture}
  \begin{tikzpicture}
    \begin{axis}[xlabel=$t$, width=8.6cm, height=5.0cm, xmin=0, xmax=200, ymin=0, ymax=25, 
      yticklabel style={
        /pgf/number format/fixed,
        /pgf/number format/precision=4
      },
      legend cell align={left},
      scaled y ticks=false
      ]
      \addplot+[blue!90!black, mark=none, very thick,
      ] table[x index=0,y index=2] {images/beads3_error.txt};
    \end{axis}
  \end{tikzpicture}
  \caption{Three-beads object. Left: Rotation angle $\angle(R_t R_0^\top)$ in degrees.
  Right: Angle $d(R_t^\mathrm{CC},R_t^\mathrm{Opt})$ in degrees between the rotation $R_t^\mathrm{CC}$ reconstructed with the common circle method and the Optimization approach $R_t^\mathrm{Opt}$ \cite{Mos25}. The average over all time steps is 10.3\,°.
  \label{fig:beads3-error}}
\end{figure}
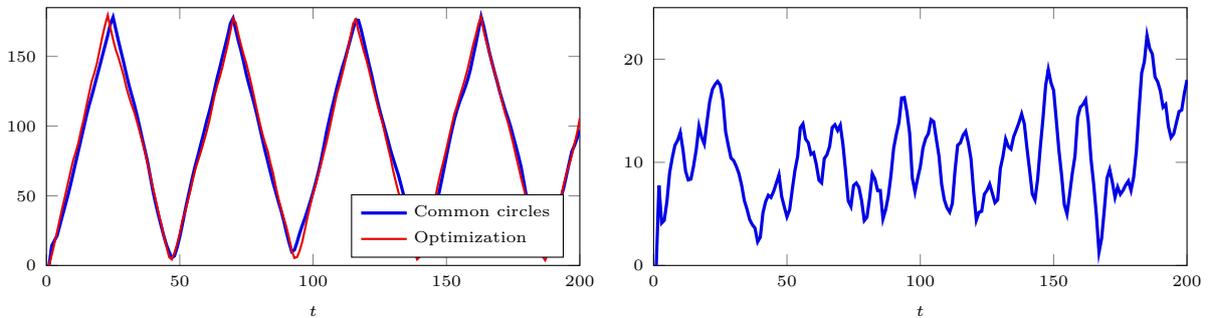

\pgfmathsetmacro{\xmin}{-5.3033}
\pgfmathsetmacro{\xmax}{5.3033}
\begin{figure}[!ht]\centering
  \foreach \x in {rec,rec-simonrot,simonrec}{
  \begin{tikzpicture}
    \begin{axis}[
      width=\modelwidth, height=\modelwidth,
      enlargelimits=false,
      scale only axis,
      axis on top,
      point meta min=1.33,point meta max=1.45,
      colorbar,colorbar style={
        width=.15cm, xshift=-0.5em, 
      },
      ticks=none,
      ]
      \addplot graphics [
      xmin=\xmin, xmax=\xmax,  ymin=\xmin, ymax=\xmax,
      ] {images/beads3-\x-xz.png};
      \path [draw, ultra thick, yscale=1, yshift=-14mm,xshift=11mm, white] (0,0) -- (2,0) node[below, midway] {\footnotesize\SI{2}{\micro\meter}};
    \end{axis}
  \end{tikzpicture}%
  }
  \caption{Reconstructed refractive index $n(\bx)$ for the Three-Beads object. Left: Rytov reconstruction with rotations from common circle method (PSNR 21.3, SSIM 0.669). 
  Center: Rytov reconstruction with rotations from \cite{Mos25} (PSNR 21.3, SSIM 0.682).
    Right: Reconstruction with the optimization approach. \label{fig:beads3-rec} 
  }
\end{figure}
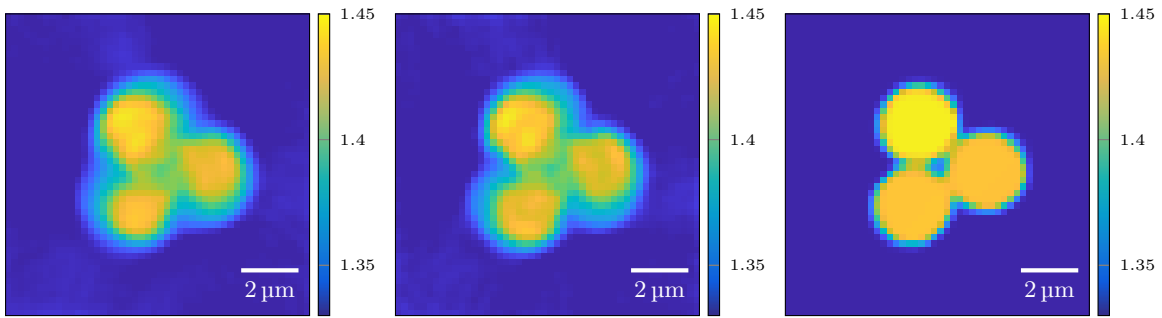

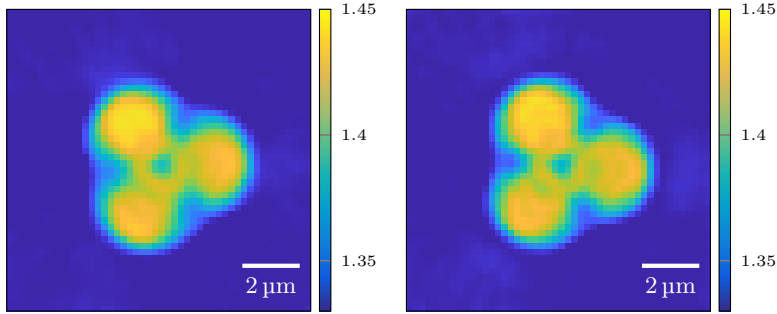
\begin{figure}[ht]\centering
\foreach \x in {rec,rec-simonrot}{
\begin{tikzpicture}
  \begin{axis}[
    width=\modelwidth, height=\modelwidth,
    enlargelimits=false,
    scale only axis,
    axis on top,
    point meta min=1.33,point meta max=1.45,
    colorbar,colorbar style={
      width=.15cm, xshift=-0.5em, 
    },
    ticks=none,
    ]
    \addplot graphics [
    xmin=\xmin, xmax=\xmax,  ymin=\xmin, ymax=\xmax,
    ] {images/beads3t-\x-xz.png};
      \path [draw, ultra thick, yscale=1, yshift=-14mm,xshift=11mm, white] (0,0) -- (2,0) node[below, midway] {\footnotesize\SI{2}{\micro\meter}};
  \end{axis}
\end{tikzpicture}
}
\caption{Reconstructed refractive index $n(\bx)$ for the Three-Beads object incorporating the reconstructed translations from \cite{Mos25}. 
Left: Rytov reconstruction with rotations from common circle method (PSNR~24.3, SSIM~0.704). 
Right: Rytov reconstruction with rotations from \cite{Mos25} (PSNR~27.0, SSIM~0.744).}
\end{figure}

All computations are performed on an Intel Core i7-10700 with 32\,GB memory.
The reconstruction of the rotations took 0.4\,s for the infinitesimal method and 33\,s for the direct method, and the Rytov reconstruction of the object about 9\,s.

\subsubsection{Neuroblastoma dataset} 
\label{sec:nb}

We have a dataset of complex-valued images $\utot_t$ for a Neuroblastoma cancer cell.
The measurements are on the grid $p \cI_N$ with the pixel size 
$p = \SI{0.1}{\micro\meter}$ and $N=250$.
We use $500$ frames, corresponding to slightly more than a full rotation.
The absolute value of $m_t$ in the first frame is depicted in \autoref{fig:nb_u}.

Whereas in the previous examples, we neglected translations of the object, 
the Neuroblastoma sample moves slightly in the image plane, 
so we estimate the shifts via the circle fitting method from Section~\ref{sec:reconstruction-steps}.
This circle fitting method produces similar results as \cite{Mos25}, see \autoref{fig:nb_translation}.

\begin{figure}\centering
  \subcaptionbox{Absolute value of preprocessed measurements $| m_t(\bx)|$ at the first frame $t=0$.
    \label{fig:nb_u}}[.45\textwidth]{
  \pgfmathsetmacro{\xmin}{-12}
  \pgfmathsetmacro{\xmax}{12}
    \vspace{0pt}
    \centering
    \begin{tikzpicture}
      \begin{axis}[
        width=6cm, height=6cm,
        enlargelimits=false,
        scale only axis,
        axis on top,
        point meta min=0,point meta max=2.891,
        colorbar,colorbar style={
          width=.15cm, xshift=-0.5em, 
        },
        ticks=none,
        ]
        \addplot graphics [
        xmin=\xmin, xmax=\xmax,  ymin=\xmin, ymax=\xmax,
        ] {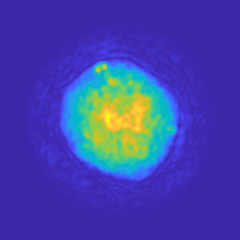};
          \path [draw, ultra thick, yscale=1, yshift=-23mm,xshift=12mm, white] (0,0) -- (5,0) node[below, midway] {\footnotesize\SI{5}{\micro\meter}};
      \end{axis}
    \end{tikzpicture}
  }\hfill
  \subcaptionbox{Translations in $x$ direction (top) and $y$ direction (bottom) reconstructed via circle fitting (blue) or the optimization approach \cite{Mos25} (red).
    \label{fig:nb_translation}}[.5\textwidth]{
  \centering
  \begin{tikzpicture}
    \begin{axis}[
    width=84mm, height=41mm, xmin=0, xmax=500, ymin=-.7, ymax=1.1, 
      yticklabel style={
        /pgf/number format/fixed,
        /pgf/number format/precision=4
      },
      legend pos={north east},
      legend cell align={left},
      scaled y ticks=false
      ]
      \addplot+[blue!90!black, mark=none, thick,
      ] table[x index=0,y index=1] {images/nb_translation.txt};
      \addplot+[red!90!black, mark=none, thick, 
      ] table[x index=0,y index=3] {images/nb_translation.txt};
      \legend{x circlefit,x optimization};
    \end{axis}
  \end{tikzpicture}
  \begin{tikzpicture}
    \begin{axis}[xlabel=$t$, width=84mm, height=40mm, xmin=0, xmax=500, ymin=-1.7, ymax=1, 
      yticklabel style={
        /pgf/number format/fixed,
        /pgf/number format/precision=4
      },
      legend pos={north east},
      legend cell align={left},
      scaled y ticks=false
      ]
      \addplot+[blue!90!black, mark=none, thick,
      ] table[x index=0,y index=2] {images/nb_translation.txt};
      \addplot+[red!90!black, mark=none, thick, 
      ] table[x index=0,y index=4] {images/nb_translation.txt};
      \legend{y circlefit,y optimization};
    \end{axis}
  \end{tikzpicture}
  }
  \caption{Visualization of Neuroblastoma dataset and reconstructed translations.}
\end{figure}

{The reconstruction of the object $f$ has the freedom of a rotation of the coordinate system, which may differ between the considered methods.
To make a fair comparison, we adjusted the reconstructed object $f_\mathrm{Opt}$ from \cite{Mos25} by a global rotation and translation, which we estimated by minimizing the least-squares fit $\sum_{\bx}|f_\mathrm{Rytov}(\bx)-f_\mathrm{Opt}(R_\mathrm{glob}\bx - \bc)|^2$ over $R_\mathrm{glob}\in\SO$ and $\bc\in\R^3$.
Hence we replace $f_\mathrm{Opt}$ by $f_\mathrm{Opt}(R_\mathrm{glob}\cdot - \bc)$.}

The reconstructed rotations are shown in \autoref{fig:nb-q}.
In \autoref{fig:nb-error}, 
we see the distance between the reconstructed rotations and the ones from \cite{Mos25} corrected for the global mismatch.
The reconstructed refractive index is shown in \autoref{fig:nb-rec}. 
Note that the reconstruction from \cite{Mos25} uses a total variation prior, which makes the image look smoother.

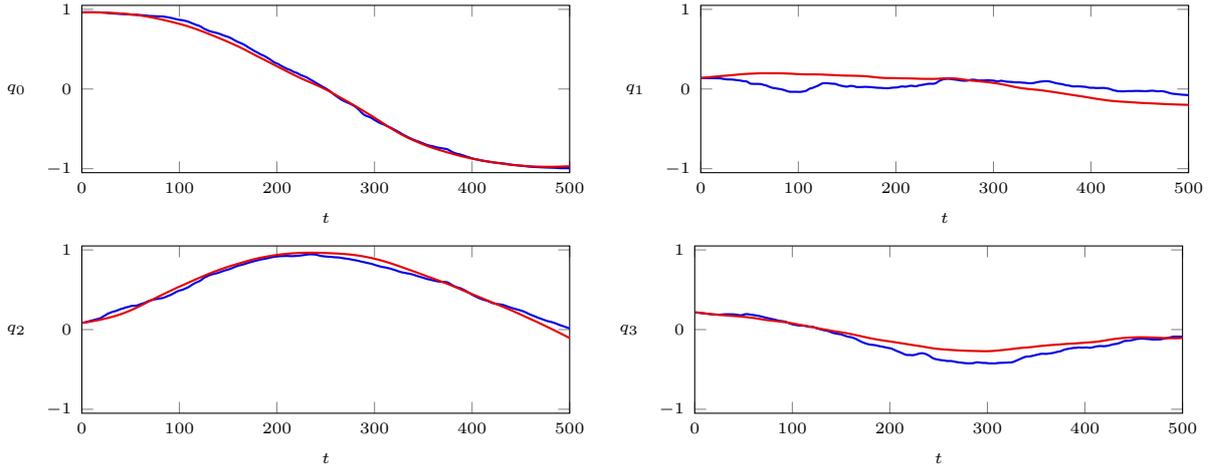
\begin{figure}[!ht]
  \centering
  \foreach \x in {1,2,3,4}{
    \begin{tikzpicture}
      \pgfmathtruncatemacro{\y}{\x + 4};
      \pgfmathtruncatemacro{\xmo}{\x - 1};
      \begin{axis}[xlabel=$t$, width=8cm, height=3.8cm, xmin=0, xmax=500, ymin=-1.05, ymax=1.05, 
        yticklabel style={
          /pgf/number format/fixed,
          /pgf/number format/precision=4
        },
        ytick={-1,0,1},
        legend pos={south east},
        legend cell align={left},
        scaled y ticks=false,
        ylabel={$q_\xmo$},
        ylabel style={rotate=-90},
        ]
        \addplot+[blue!90!black, mark=none, thick,
        ] table[x index=0,y index=\x] {images/nb_quaternions.txt};
        \addplot+[red!90!black, mark=none, thick,
        ] table[x index=0,y index=\y] {images/nb_quaternions.txt};
      \end{axis}
    \end{tikzpicture}\hfill
  }
  \caption{Reconstructed rotations for the Neuroblastoma. 
    The four plots are the components of the quaternions.  
    Blue: Common circle method. Red: Optimization approach \cite{Mos25}.}
    \label{fig:nb-q}
\end{figure}

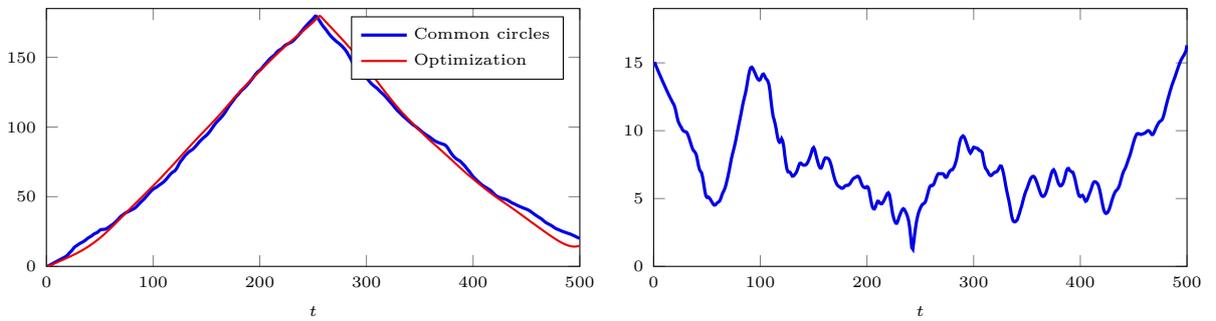
\begin{figure}[!ht]
  \begin{tikzpicture}
    \begin{axis}[xlabel=$t$, width=8.6cm, height=5.0cm, xmin=0, xmax=500, ymin=0, ymax=185, 
      yticklabel style={
        /pgf/number format/fixed,
        /pgf/number format/precision=4
      },
      legend pos={north east},
      legend cell align={left},
      scaled y ticks=false
      ]
      \addplot+[blue!90!black, mark=none, very thick,
      ] table[x index=0,y index=2] {images/nb_angle.txt};
      \addplot+[red!90!black, mark=none, thick,
      ] table[x index=0,y index=3] {images/nb_angle.txt};
      \legend{Common circles,Optimization};
    \end{axis}
  \end{tikzpicture}
  \begin{tikzpicture}
    \begin{axis}[xlabel=$t$, width=8.6cm, height=5.0cm, xmin=0, xmax=500, ymin=0, ymax=19, 
      yticklabel style={
        /pgf/number format/fixed,
        /pgf/number format/precision=4
      },
      legend cell align={left},
      scaled y ticks=false
      ]
      \addplot+[blue!90!black, mark=none, very thick,
      ] table[x index=0,y index=1] {images/nb_error_adjusted.txt};
    \end{axis}
  \end{tikzpicture}
  \caption{Reconstructed rotations for the Neuroblastoma. Left: Rotation angle $d(R_t, R_0)$ in degrees.
    Right: Angle $d(R_t^\mathrm{CC},R_t^\mathrm{Opt})$ in degrees between the rotation $R_t^\mathrm{CC}$ reconstructed with the common circle method and $R_t^\mathrm{Opt}$ with the Optimization approach \cite{Mos25},
    the average over all time steps is 7.7\,°.  \label{fig:nb-error}} 
\end{figure}

\begin{figure}[!ht]\centering
  \pgfmathsetmacro{\xmax}{10.6066}
  \pgfmathsetmacro{\cmin}{1.33}
  \pgfmathsetmacro{\cmax}{1.375}
  
  \begin{subfigure}{\textwidth}\centering
  \foreach \x in {xy,xz,yz}{
    \begin{tikzpicture}
      \begin{axis}[
        width=\modelwidth, height=\modelwidth,
        enlargelimits=false,
        scale only axis,
        axis on top,
        point meta min=\cmin,point meta max=\cmax,
        colorbar,colorbar style={
          width=.15cm, xshift=-.5em, 
        ytick={1.33,1.34,1.35,1.36,1.37},
        },
        ticks=none,
        ]
        \addplot graphics [
        xmin=-\xmax, xmax=\xmax,  ymin=-\xmax, ymax=\xmax,
        ] {images/nb-simonrec-adjusted-\x.png};
      \path [draw, ultra thick, yscale=1, yshift=-14mm,xshift=11mm, white] (0,0) -- (2,0) node[below, midway] {\footnotesize\SI{2}{\micro\meter}};
      \end{axis}
    \end{tikzpicture}\hfill
  }
  \caption{Reconstruction by optimization approach \cite{Mos25} (taken as ground truth). 
  \label{fig:nb-rec-simonrot-adjusted}}
  \end{subfigure}
  
  \begin{subfigure}{\textwidth}\centering
  \foreach \x in {xy,xz,yz}{
  \begin{tikzpicture}
    \begin{axis}[
      width=\modelwidth, height=\modelwidth,
      enlargelimits=false,
      scale only axis,
      axis on top,
      point meta min=\cmin,point meta max=\cmax,
      colorbar,colorbar style={
        width=.15cm, xshift=-.5em, 
        ytick={1.33,1.34,1.35,1.36,1.37},
      },
      ticks=none,
      ]
      \addplot graphics [
      xmin=-\xmax, xmax=\xmax,  ymin=-\xmax, ymax=\xmax,
      ] {images/nb-rec-\x.png};
      \path [draw, ultra thick, yscale=1, yshift=-14mm,xshift=11mm, white] (0,0) -- (2,0) node[below, midway] {\footnotesize\SI{2}{\micro\meter}};
    \end{axis}
  \end{tikzpicture}\hfill
  }
  \caption{Rytov reconstruction with rotations from common circle method and translations from the circle fitting (PSNR 37.8, SSIM 0.900).}
  \end{subfigure}
  
  \begin{subfigure}{\textwidth}\centering
    \foreach \x in {xy,xz,yz}{
      \begin{tikzpicture}
        \begin{axis}[
          width=\modelwidth, height=\modelwidth,
          enlargelimits=false,
          scale only axis,
          axis on top,
          point meta min=\cmin,point meta max=\cmax,
          colorbar,colorbar style={
            width=.15cm, xshift=-.5em, 
        ytick={1.33,1.34,1.35,1.36,1.37},
          },
          ticks=none,
          ]
          \addplot graphics [
          xmin=-\xmax, xmax=\xmax,  ymin=-\xmax, ymax=\xmax,
          ] {images/nb-rec-simonrot-adjusted-\x.png};
      \path [draw, ultra thick, yscale=1, yshift=-14mm,xshift=11mm, white] (0,0) -- (2,0) node[below, midway] {\footnotesize\SI{2}{\micro\meter}};
        \end{axis}
      \end{tikzpicture}\hfill
    }
    \caption{Rytov reconstruction with rotations and translations from \cite{Mos25} (PSNR 38.5, SSIM 0.902).}
  \end{subfigure}
  
  \caption{Reconstructed refractive index $n(\bx)$ for the Neuroblastoma with different methods. 
  Columns are sections through the coordinate planes (left to right: xy-plane, xz-plane, yz-plane).
  \label{fig:nb-rec}}
\end{figure}
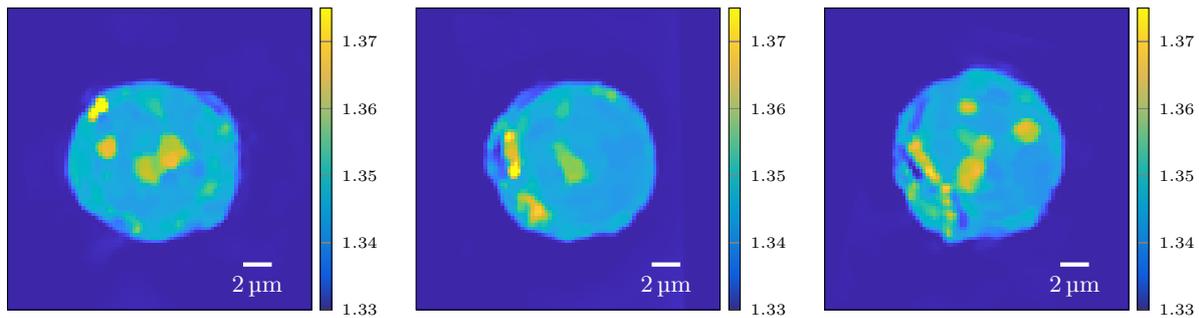
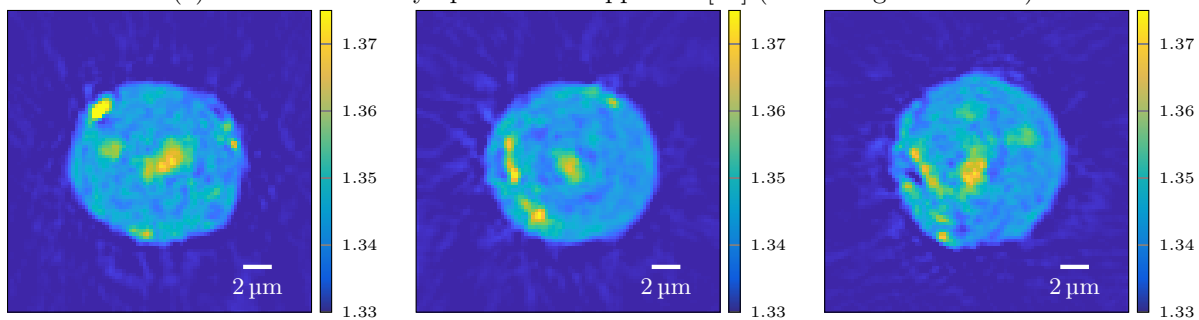
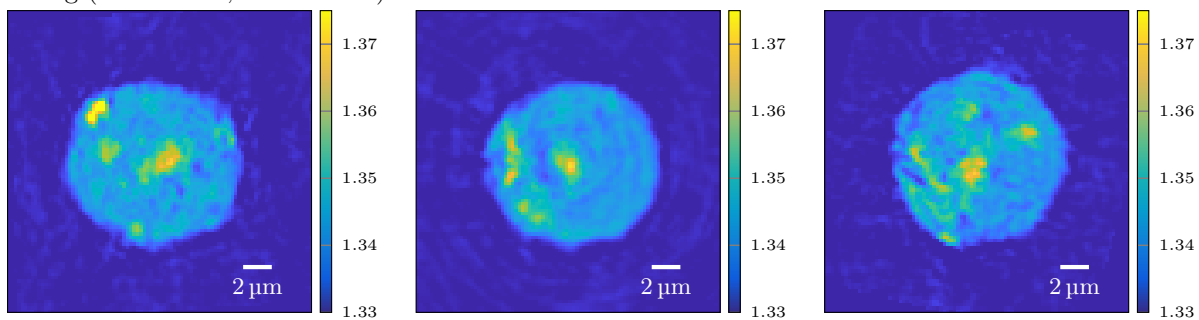

\section{Conclusions}
\label{sec:conclusions}

We have demonstrated the practical application of the common circle method for motion detection in optical diffraction tomography using experimental data. 
By introducing temporal regularization and ensuring consistency across time steps, 
we achieved stable motion estimation, 
which allowed for the accurate 3D reconstruction of the refractive index in biological samples.

Moving forward, we aim to automate the reconstruction pipeline and develop a hybrid approach combining the speed and initialization-free nature of the common circle method with the precision of full optimization technique \cite{Mos25}.
Such a framework would be beneficial for imaging larger objects efficiently.
In order to enhance noise robustness,
we want to utilize sliced optimal transport \cite{BeiQue22,nguyen2025introduction,QueBueSte2024},
extending methods for rigid alignment of planes \cite{shi2025fast}
and synchronizing rotations \cite{BiArSiGu20}.

\subsection*{Acknowledgments}
Funding by the DFG under the SFB ``Tomography Across the Scales'' (STE 571/19-1, project number: 495365311) is gratefully acknowledged. 
Moreover, MKL and MRM are supported by the Austrian Science Fund (FWF),
with SFB 10.55776/F68 ``Tomography Across the Scales'', project F6806-N36 Inverse Problems in Imaging of Trapped Particles (MRM).
For the purpose of open access, the authors have applied a CC BY public copyright license to any authors-accepted manuscript version arising from this submission.

We thank Judith Hagenbuchner and Michael Ausserlechner (both Dept.\ of Pediatrics I and 3D Bioprinting Lab, Medical University Innsbruck) for providing us with the neuroblastoma sample. 
Furthermore, we thank Peter Elbau, Otmar Scherzer, Denise Schmutz (all University of Vienna) and Gabriele Steidl (TU Berlin) for fruitful conversations about the common circles method.

\small
\bibliographystyle{abbrvurl}
\bibliography{nocsc}

\end{document}